  \documentclass[12pt,reqno]{amsart}
  \usepackage{latexsym} 
  \usepackage[all]{xy}
  \usepackage{amsfonts} 
  \usepackage{amsthm} 
  \usepackage{amsmath} 
  \usepackage{amssymb}
  \usepackage{pifont}  
  \usepackage{enumerate}
  \xyoption{2cell}
  \usepackage{amscd}

 \usepackage{tikz}
\usetikzlibrary{arrows}

  \def\<{{\langle}} 
  \def\>{{\rangle}}

  \def\note#1{{}}

  \def\note#1{} 

  \def\rhom#1#2#3{{{\rm Hom}\sb{#1}(#2,#3)}}

  \def\beq{\begin{equation}} 
  \def\eeq{\end{equation}}

  \def\ot{{\otimes}} 
   
  \def\Hom{\mbox{\rm Hom}\,}

 \def\gk{\mathrm{GKdim}}

  \newcounter{zlist} 
  \newenvironment{zlist}{\begin{list}{(\arabic{zlist})}{ 
  \usecounter{zlist}\leftmargin2.5em\labelwidth2em\labelsep0.5em 
  \topsep0.6ex
  \parsep0.3ex plus0.2ex minus0.1ex}}{\end{list}}

  \newcounter{blist} 
  \newenvironment{blist}{\begin{list}{(\alph{blist})}{ 
  \usecounter{blist}\leftmargin2.5em\labelwidth2em\labelsep0.5em 
  \topsep0.6ex 
  \parsep0.3ex plus0.2ex minus0.1ex}}{\end{list}} 

  \newcounter{rlist}

\def\stac#1{\raise-.2cm\hbox{$\stackrel{\displaystyle\otimes}{\scriptscriptstyle{#1}}$}}

\def\cten#1{\raise-.2cm\hbox{$\stackrel{\displaystyle\widehat{\otimes}}
{\scriptscriptstyle{#1}}$}}

  \def\Label#1{\label{#1}\ifmmode\llap{[#1] }\else 
  \marginpar{\smash{\hbox{\tiny [#1]}}}\fi} 
  \def\Label{\label}

  \newtheorem{proposition}{Proposition}[section]
  \newtheorem{lemma}[proposition]{Lemma} 
  \newtheorem{corollary}[proposition]{Corollary} 
  \newtheorem{theorem}[proposition]{Theorem} 

\theoremstyle{definition} 
  \newtheorem{definition}[proposition]{Definition}
    
  \newtheorem{example}[proposition]{Example}

  \theoremstyle{remark} 
  \newtheorem{remark}[proposition]{Remark}

  \newcounter{c} 
   
  \newcommand{\etyk}[1]{\vspace{-7.4mm}$$\begin{equation}\Label{#1} 
  \addtocounter{c}{1}} 
  \renewcommand{\]}{\ifnum \value{c}=1 $$\else \end{equation}\fi} 
\def\ot{\otimes}

\def\FF{{\mathbb F}}

\def\NN{{\mathbb N}}

\def\ZZ{{\mathbb Z}}

\newcommand{\Cc}{\mathcal{C}}

\newcommand{\Vv}{\mathcal{V}}

\def\*C{{}^*\hspace*{-1pt}{\Cc}}

\def\text#1{{\rm {\rm #1}}}

 \def\v{\mathbf{r}}

 \def\1{\mathbf{1}}

\renewcommand{\Hom}{\mathrm{Hom}}
\numberwithin{equation}{section}
\def\v{\mathsf{v}}
 \newcommand{\volumer}{\v_R}
\newcommand{\volumes}{\v_S}
\newcommand{\topform}{\v}

\makeatletter
\newenvironment{proofof}[1]{\par
  \pushQED{\qed}%
  \normalfont \topsep6\p@\@plus6\p@\relax
  \trivlist
  \item[\hskip\labelsep
        \sl
 \indent   Proof of #1\@addpunct{.}]\ignorespaces
}{%
  \popQED\endtrivlist\@endpefalse
}
\makeatother

\begin{document}

\title{Differential smoothness of skew polynomial rings}

\author{Tomasz Brzezi\'nski}
 \address{ Department of Mathematics, Swansea University, 
  Swansea SA2 8PP, U.K.\ \newline 
\indent Department of Mathematics, University of Bia{\l}ystok, K.\ Cio{\l}kowskiego  1M,
15-245 Bia\-{\l}ys\-tok, Poland} 
  \email{T.Brzezinski@swansea.ac.uk}   
  
  \author{Christian Lomp}
 \address{Departamento de Matem\'atica, Faculdade de Ci\^encias, Universidade do Porto, 4169-007 Porto, Portugal}
 \email{clomp@fc.up.pt}
    \date{\today} 

\subjclass[2010]{16S38; 16S36, 58B34, 16E45} 
 \keywords{Differentially smooth algebra; skew polynomial ring}
 
\begin{abstract}
It is shown that, under some natural assumptions, the tensor product of differentially smooth algebras and the skew-polynomial rings over differentially smooth algebras are differentially smooth.
\end{abstract}
\maketitle

\section{Introduction}
The study of smoothness of algebras goes back at least to Grothendieck's EGA. The concept of a formally smooth commutative (topological) algebra introduced in there \cite[D\'efinition~19.3.1]{Gro:etu} was later extended to the non-commutative case by Schelter in \cite{Sch:smo}. An algebra is formally smooth if and only if the kernel of the multiplication map is projective as an bimodule. As argued by Schelter himself, this notion arose as a replacement of a far too general definition based on the finiteness of the global dimension. Although it plays an important role in non-commutative geometry (see e.g.\ \cite{CunQui:alg}, where such algebras are termed {\em quasi-free}), the notion of formal smoothness seems to be too restrictive. The too crude notion of smoothness based on the finiteness of the global dimensions was refined in \cite{StaZha:hom}, where a Noetherian algebra was said to be smooth provided that it had a finite global dimension equal to the homological dimension of all its simple modules. From the homological perspective probably most satisfying is the notion of {\em homological smoothness} introduced in \cite{Van:rel}: an algebra is homologically smooth provided it admits a finite resolution by finitely generated projective bimodules. Algebras of this kind display a Poincar\'e type duality between Hochschild homology and cohomology, and retain many properties characteristic of co-ordinate algebras of  smooth varieties (see e.g.\ \cite{Kra:Hoc}, where this last point is strongly argued for).

A different and more constructive approach to smoothness of algebras was taken in \cite{BrzSit:smo}. In this approach the smoothness of an algebra $A$ is related to the existence of a specific differential graded algebra (with $A$ as the degree-zero part) whose size is aligned with the rate of growth of $A$ measured by the Gelfand-Kirillov dimension, and which satisfies a strict version of the Poincar\'e duality in terms of an isomorphism with the corresponding complex of integral forms \cite{BrzElK:int} (see Section~\ref{sec.prelim} for precise definition). In view of this direct use of differential graded algebras this kind of smoothness is referred to as {\em differential smoothness}. The main advantage of this approach is its concreteness: a differentially smooth algebra comes equipped with a well-behaved differential structure and with the precisely defined concept of integration. Examples of differentially smooth algebras include the coordinate algebras of the quantum group $SU_q(2)$, the quantum 2-sphere (see \cite{BrzElK:int}), the non-commutative pillow algebra, the quantum cone algebras (see \cite{BrzSit:smo}), the quantum polynomial algebras (see \cite{KarLom:int}), and  Hopf algebra domains of Glefand-Kirillov dimension 2 that are not PI (see \cite{Brz:dif}). Although many of these examples are known to be also homologically smooth, the relationship between the differential and other types of smoothness is not clear yet. 

At the root of difficulties with comparing differential and other types of smoothness is the constructive nature of the former, which prevents one from using functorial or just existential arguments. In this paper we make a few steps toward resolving such difficulties and present two general constructions which lead from differentially smooth to differentially smooth algebras. First, we show that -- under some natural assumptions on differential structures and algebras -- the tensor product of differentially smooth algebras is differentially smooth. This allows one to deduce quickly smoothness of polynomial and Laurent polynomial rings without necessity of constructing specific differential structure (it suffices to have such a structure for polynomials in one variable). Second, again under some natural assumptions, we prove that the skew-polynomial rings over a smooth algebra are smooth.

\section{Preliminaries}\label{sec.prelim}
Let $\FF$ be a field. By a {\em differential calculus} over an $\FF$-algebra $R$ we mean a differential graded algebra $(\Omega R, d)$ (i.e.\ a graded algebra with the degree-one square-zero linear map $d:\Omega R \to \Omega R$ satisfying the graded Leibniz rule) such that:
\begin{blist}
\item $\Omega R = \bigoplus_{n\in \NN} \Omega^n R$, i.e.\ it  is non-negatively graded, and $\Omega^0R =R$,
\item For all $n\in \NN$, 
$$
\Omega^n R = R\underbrace{d(R)d(R)\cdots d(R)}_{\mbox{$n$-times}}.
$$
\end{blist}
The requirement (b) is called the {\em density condition}. 
A differential calculus  $(\Omega R, d)$  over $R$ is said to be {\em connected}, provided $\ker \left(d\mid_R\right) = \FF$. It is said to have {\em dimension $N$} or to be {\em $N$-dimensional} provided 
$$
\Omega^N R \neq 0 \qquad \mbox{and} \qquad  \Omega^nR =0 ,\quad \mbox{for all $n>N$}.
$$
An $N$-dimensional differential calculus  $(\Omega R, d)$ over $R$ is said to {\em admit a  volume form}, provided $\Omega^N R$ is isomorphic to $R$ as a both left and right $R$-module (but not necessarily as an $R$-bimodule). Any free generator $\v$ of $\Omega^N R$  as a right and left $R$-module is referred to as a {\em volume form}. Associated to a volume form $\v$ are two maps:
\begin{blist}
\item the right $R$-module {\em co-ordinate isomorphism}:
\begin{equation}\label{co-ord}
\pi_\v: \Omega^N R \to R, \qquad \pi_\v(\v r) =r;
\end{equation}
\item the {\em twisting} algebra automorphism:
\begin{equation}\label{twist.auto}
\theta_\v: R \to R, \qquad  r\mapsto \pi_\v(r\v).
\end{equation}
\end{blist}
Note that the admittance of a volume form does not automatically imply the existence of a volume form. To any right $R$-linear homomorphism $\varphi: \Omega^n R\to R$ we associate a family of right $R$-module maps
\begin{equation}\label{ell.gen}
\ell^k_\varphi :{\Omega^k R} \to  \rhom R {\Omega^{n-k}R} R, \quad \gamma \mapsto [\gamma'\mapsto \varphi(\gamma   \gamma')], \qquad k\in \{0,1,\ldots , n\}.
\end{equation}
The right $R$-multiplication on the space of right $R$-linear maps  $\rhom R {\Omega^{k}R} R$ is defined by $(\psi r)(\gamma) = \psi(r\gamma)$, for all $\psi \in \rhom R {\Omega^{k}R} R$, $r\in R$ and $\gamma\in \Omega^kR$. The maps $\ell^k_{\pi_\v}$ associated to a volume form co-ordinate isomorphism \eqref{co-ord} are $R$-bimodule homomorphisms, provided the left $R$-multiplication on $\rhom R {\Omega^{k}R} R$ is defined via the twisting automorphism $\theta_\v$, 
$$
(r \psi)(\gamma) = \theta_\v(r)\psi(\gamma).
$$
Following \cite{BrzSit:smo} a differential calculus with a volume $N$-form $\v$ is said to be {\em integrable} provided all bimodule homomorphisms $\ell^k_{\pi_\v}$, $k\in \{0,1,\ldots ,N\}$ are invertible. This is equivalent to the the existence of a complex of integrable forms \cite{BrzElK:int} isomorphic to $(\Omega R, d)$ (see \cite[Theorem~2.2]{BrzSit:smo}). To relieve the notation we will write $\ell^k_\v$ or simply $\ell^k$ for $\ell^k_{\pi_\v}$. For the future use we thus record that if $\v \in \Omega^NR$ is a volume form, then
\begin{equation}\label{ell}
\ell^k_\v  :{\Omega^k R} \to  \rhom R {\Omega^{N-k}R} R, \quad \gamma \mapsto [\gamma'\mapsto \pi_\v(\gamma   \gamma')], \qquad k\in \{0,1,\ldots , N\}.
\end{equation}

Given an affine $\FF$-algebra $R$ with generating subspace $\Vv$, let us write $\Vv(n)$ for the subspace of $R$ spanned by 1 and all words in generators of $R$ of length at most $n$.  The {\em Gelfand-Kirillov dimension} of $R$ is a real number defined as
\begin{equation}\label{GK}
\gk(A) := \inf \{ t\; |\; \dim \Vv(n) \leq n^t, \, n \gg 0\},
\end{equation}
if it exists, and is defined as infinity otherwise. The Gelfand-Kirllov dimension of an arbitrary $\FF$-algebra $R$ is by definition the supremum of the Gelfand-Kirillov dimensions of affine $\FF$-subalgebras of $R$ (see \cite[8.1.16]{McCRob:Noe}).
Although it is not generally true that the Gelfand-Kirillov dimension of the tensor product of two algebras $R$, $S$ is equal to the sum of their (finite) Gelfand-Kirillov dimensions, it is, however, the case that if $\gk(R) \leq 2$ or $\gk(S) \leq 2$, then
$$
\gk(R\ot S) = \gk (R) +\gk(S);
$$
see \cite[Proposition 3.12]{KraLen:gro}. We refer the reader to  \cite{KraLen:gro} or\cite[Chapter~8]{McCRob:Noe} for a detailed discussion of the Gelfand-Kirillov dimension, which, in the case of a commutative Noetherian algebra is a very good measure of geometric dimension of the underlying affine space. 

The version of smoothness studied in the present text is recalled in the following
\begin{definition}[\cite{BrzSit:smo}]
An affine algebra $R$ of integer Gelfand-Kirillov dimension $N$ is said to be {\em differentially smooth}, if there exists a connected, $N$-dimensional, integrable differential calculus on $R$.
\end{definition}

Let $R$ be an algebra and $\sigma$ an algebra automorphism of $R$. By a {\em skew-polynomial ring over $R$} we mean the algebra $R[z;\sigma]$ generated additionally by $z$ and the relations $zr = \sigma(r)z$, for all $r\in R$. Similarly the Laurent skew-polynomial ring $R[z^{\pm 1};\sigma]$ is defined. As was the case for tensor product algebras, it is not generally true that $\gk(R[z;\sigma]) = \gk(R) +1$ (see  \cite[Example 8.2.16]{McCRob:Noe}). The equality holds, whenever $\sigma$ is {\em locally algebraic}, i.e.\ if for all $r\in R$, the set $\{\sigma^n(r)\; |\; n\in \NN\}$ is contained in a finite dimensional subspace of $R$  (see \cite[Proposition~1]{LerMat:Gel}).

As we will often make statements that apply equally well to the skew-polynomial and the Laurent skew-polynomial rings, we reserve the symbol $R[z^\bullet;\sigma]$ to denote either $R[z;\sigma]$ or $R[z^{\pm 1};\sigma]$.

\section{Differential smoothness of the tensor product of algebras}\label{sec.tensor}
The aim of this section is to prove that, under some mild and geometrically natural assumptions, tensor product of integrable differential calculi on two algebras gives an integrable calculus on the tensor product algebra.

Suppose that $(\Omega R, d_R)$, where $\Omega R = \bigoplus_{k=0}^N \Omega^k R$, is an $N$-dimensional differential calculus on an $\FF$-algebra $R$, and that $(\Omega S, d_S)$, where 
 $\Omega S = \bigoplus_{k=0}^M \Omega^k S$, is an $M$-dimensional differential calculus on an $\FF$-algebra $S$. Consider $T:=R\otimes S$ and 
$$ 
\Omega T  := \Omega R \otimes \Omega S  =  \bigoplus_{n=0}^{N+M} \left( \underbrace{ \sum_{i=0}^k \Omega^i R \otimes \Omega^{k-i}S}_{=:\Omega^k T}\right).
$$
Components $\Omega^iR$, resp.\ $\Omega^jS$, are considered to be zero if $i$ or $j$ are not within their limits.
$\Omega T$ becomes a differential graded algebra, with graded multiplication defined as 
\begin{equation}\label{wedge.T} (\omega \otimes \nu)   (\omega' \otimes \nu) = (-1)^{|\nu||\omega'|} \omega  \omega' \otimes \nu  \nu',
\end{equation}
for homogeneous elements $\omega, \omega', \nu, \nu'$,
and extended differential $d_T$ of $\Omega T$ defined by
\begin{equation}\label{dT}
 d_T(\omega \otimes \nu) := \omega \otimes d_S(\nu) \: + (-1)^i d_R(\omega)\otimes \nu, 
 \end{equation}
for all $\omega \in \Omega^iR$ and  $\nu\in\Omega^jS$.
By the density condition
$$
\omega = \sum_t r_0^td_R(r_1^t) \cdots   d_R(r_i^t), \qquad \nu = \sum_u s_0^ud_S(s_1^u)  \cdots   d_S(s_j^u),
$$
hence, in view of \eqref{wedge.T} and \eqref{dT},
$$
\omega \ot \nu  = \sum_{t,u} (r_0^t\ot s_0^u)d_T(r_1^t\ot 1)  \cdots   d_T(r_i^t\ot 1)   d_T(1\ot s_1^u)  \cdots   d_T(1\ot s_j^u).
$$
Therefore,  the differential graded algebra $(\Omega T, d_T)$ is a differential calculus on $T$.

\begin{proposition}\label{prop.tensor}
Let $R$ and $S$ be algebras with integrable differential calculi $(\Omega R,d_R)$ and $(\Omega S, d_S)$. Suppose that $\Omega R$ is a finitely generated projective right $R$-module and that  $\Omega S$ is a finitely generated projective right $S$-module. Then $(\Omega R \otimes \Omega S, d)$ is an integrable differential calculus for $R\otimes S$. 
\end{proposition}

\begin{proof}
We write $T:= R\ot S$ and assume that $(\Omega^R,d_R)$ has dimension $N$ and $(\Omega^S,d_S)$ has dimension $M$. Note that for  homogeneous $\omega \otimes \nu \in \Omega^iR\otimes \Omega^{k-i}S$ and $\omega' \otimes \nu' \in \Omega^jR\otimes \Omega^{k'-j}S$ we have
$$ (\omega \otimes \nu)   (\omega' \otimes \nu') = (-1)^{(k-i)j}\underbrace{\omega\,   \omega'}_{\in \Omega^{i+j}R} \otimes \underbrace{\nu\,   \nu'}_{\in \Omega^{k+k'-i-j}S} \in \Omega^{k+k'}T.$$

Since $\Omega^{i+j}R= 0$ for $i+j>N$ and $\Omega^{N+M-i-j}S=0$ for $i+j<N$, we have for all $k$:
$$
\left( \Omega^iR\otimes \Omega^{k-i}S\right)   \left(\Omega^jR \otimes \Omega^{N+M-k-j}S\right) = 0, \qquad \: \forall j\neq N-i.
$$
This means in particular 
\begin{equation} \label{components} 
\left( \Omega^iR\otimes \Omega^{k-i}S\right)   \Omega^{N+M-k}T = \left( \Omega^iR\otimes \Omega^{k-i}S\right)   \left( \Omega^{N-i}R \otimes \Omega^{M-(k-i)}S\right),
\end{equation}
for all $i\leq k$.

Since $\Omega^{N+M}T=\Omega^NR   \otimes    \Omega^MS$,  for all $\varphi_1 \in \Hom_R(\Omega^{N}R, R)$ 
and $\varphi_2 \in \Hom_S(\Omega^{M}S, S)$,   $\varphi := \varphi_1 {\otimes} \varphi_2 \in \Hom_T(\Omega^{N+M}T, T)$, and so we can
consider the maps \eqref{ell.gen},
$ \ell^k_\varphi  : \Omega^kT \rightarrow \Hom_T(\Omega^{N+M-k}T, T)$. 
By equation (\ref{components}), 
$$
\ell^k_\varphi\left( {\Omega^iR \otimes \Omega^{k-i}S}\right)\subseteq \Hom_T({\Omega^{N-i}R \otimes \Omega^{M-(k-i)}S}, T).
$$
 In particular
$$
\ell_\varphi^k (\omega \otimes \nu)(\omega' \otimes \nu') = 
\varphi_1(\omega  \omega') \otimes \varphi_2(\nu\otimes \nu') = \ell^i_{\varphi_1}(\omega)(\omega') \otimes \ell^{k-i}_{\varphi_2}(\nu)(\nu'),
$$
for all $ \omega\in \Omega^iR,  \omega'\in \Omega^{N-i}R$,  $\nu\in \Omega^{k-i}S, \nu'\in\Omega^{M-k+i}S$, 
where $\ell^i_{\varphi_1}$ and $\ell^{k-i}_{\varphi_2}$ are defined by \eqref{ell.gen}.
Identifying  $\ell^i_{\varphi_1}\otimes \ell^{k-i}_{\varphi_2} \in \Hom_R(\Omega^iR, R) \otimes \Hom_S(\Omega^{k-i}S,S)$ with an element of $\Hom_T(\Omega^iR\otimes \Omega^{k-i}S, T)$,  we obtain $\ell^k_\varphi = \sum_{i=0}^k \ell^i_{\varphi_1} {\otimes} \ell^{k-i}_{\varphi_2},$ since $\Omega^kT = \bigoplus_{i=0}^k \Omega^iR \otimes \Omega^{k-i}S$.

Suppose that $\volumer \in \Omega^NR$ and  $\volumes \in \Omega^MS$ are volume forms with corresponding co-ordinate isomorphisms $\pi_{\volumer}:\Omega^NR \rightarrow R$ and $\pi_{\volumes}:\Omega^NS \rightarrow S$. Then $\v =\volumer \otimes \volumes$ is a volume form for $\Omega^{N+M}T$ with isomorphism $\pi_\v = \pi_{\volumer} \otimes \pi_{\volumes}$. We have already seen that $\ell^k_{\v} = \sum_{i=0}^k  \ell^i_{\volumer} \otimes \ell^{k-i}_{\volumes}$, for all $0<k<N+M$.

By assumption the maps $\ell^i_{\volumer}$ and $\ell^j_{\volumes}$ are bijective for all $0\leq i,j \leq k$. Hence also 
$$
\ell^i_{\volumer}\otimes \ell^{k-i}_{\volumes}: \Omega^iR \otimes \Omega^{k-i}S \longrightarrow \Hom_R(\Omega^{n-i}R,R)\otimes  \Hom_S(\Omega^{m-(k-i)}S,S),
$$
is bijective for all $0\leq i \leq k$.

If $\Omega^{N-i}R$ and $\Omega^{M-(k-i)}S$ are finitely generated projective as right $R$-modules, respectively as right $S$-modules, then by \cite[15.9]{Wis:mod},
$$
\Hom_R(\Omega^{N-i}R,R)\otimes  \Hom_S(\Omega^{M-(k-i)}S,S) =\Hom_T( \Omega^{N-i}R \otimes  \Omega^{M-(k-i)}S, T),
$$
and the maps $\ell^i_{\volumer}\otimes \ell^{k-i}_{\volumes}$ between $\Omega^iR \otimes \Omega^{k-i}S$ and $\Hom_T( \Omega^{N-i}R \otimes  \Omega^{M-(k-i)}S, T)$ are bijections.
Thus
$$ \ell^k_\v:\underbrace{\bigoplus_{i=0}^k \Omega^iR \otimes \Omega^{k-i}S}_{\Omega^kT} \xrightarrow{\sum \ell^i_{\volumer} \otimes \ell^{k-i}_{\volumes}} 
\underbrace{\bigoplus_{i=0}^k  \Hom_R\left(\Omega^{N-i}R,R\right)\otimes  \Hom_S\left(\Omega^{M-(k-i)}S,S\right)}_{  
\Hom_T\left(\bigoplus_{i=0}^k  \Omega^{N-i}R \otimes \Omega^{M-(k-i)}S, T\right)}$$
is a bijection.
Since 
$$ \bigoplus_{i=0}^k  \Omega^{N-i}R \otimes \Omega^{M-(k-i)}S
= \bigoplus_{j=0}^{N-k}  \Omega^jR \otimes \Omega^{N+M-k-j}S
= \Omega^{N+M-k}T,$$
where components of  $\Omega R$ or $\Omega S$ are zero if their degrees are not within the limits, we eventually conclude that  $\ell^k_\pi$ is a bijection between $\Omega^kT$ and $\Hom_T(\Omega^{N+M-k}T,T)$.
\end{proof}

Proposition~\ref{prop.tensor} yields
\begin{corollary}\label{cor.tensor}
If $R$ and $S$ are differentially smooth algebras with respect to calculi which are finitely generated projective as right modules and 
$$
\gk(R\ot S) = \gk(R)\ + \gk(S),
$$
then the tensor product algebra $R\ot S$ is differentially smooth.
\end{corollary}
\begin{proof} We only need to check whether the connectedness of $\Omega R$ and $\Omega S$ implies the connectedness of $\Omega T$. Let $x= \sum_{i,j} \alpha_{ij}r_i\ot s_j \in \ker d_T \subseteq R\ot S$, where the sets $\{r_i\}$ and $\{s_j\}$ are linearly independent and $\alpha_{ij}\in \FF$. If $\Omega R$ is connected, then in view of the definition \eqref{dT}, $\sum_i \alpha_{ij}r_i$ is a scalar multiple of 1, for all $j$, i.e.\ $x = \sum_j 1\ot \beta_j s_j$, for some scalars $\beta_j$. If, furthermore $\Omega S$ is connected, then  the definition \eqref{dT} implies that  $\sum_j \beta_js_j$ is a scalar multiple of $1$, hence $x$ is a scalar multiple of $1\otimes 1$.
Therefore, $\Omega T$ is connected. The assertion then follows by Proposition~\ref{prop.tensor}.
\end{proof}

\begin{corollary}\label{cor.poly}
Let $R$ be a differentially smooth algebra with respect to a differential calculus $\Omega R$ that is finitely generated and projective over $R$. Then extensions of the form $R[x_1, \ldots, x_n, y_1^{\pm 1}, \ldots, y_m^{\pm 1}]$ are also  differentially smooth.
\end{corollary}

\begin{proof}
Both the polynomial algebra $\FF[x]$ and the Laurent polynomial algebra $\FF[y^{\pm 1}]$ have Gelfand-Kirillov dimension one and are smooth. In the case of $\FF[x]$ a connected one-dimensional integrable differential calculus is (freely as a module) generated  by the volume one-form $\v = d(x)$ and the associated twisting automorphism  $\theta_\v(f(x)) = f(qx)$, where $q$ is any non-zero element of $\FF$ (this determines fully the structure of $\Omega\FF[x] = \FF[x] \oplus \Omega^1\FF[x]$). In the case of $\FF[y^{\pm 1}]$, the volume form can be chosen as $\v = y^{-1}d(y)$ and then $\theta_v$ is the identity map. Since 
$$
R[x_1, \ldots, x_n, y_1^{\pm 1}, \ldots, y_m^{\pm 1}] = R\ot \FF[x_1]\ot \ldots\ot \FF[x_n]\ot \FF[y_1^{\pm 1}]\ot \ldots\ot \FF[y_m^{\pm 1}],
$$
and $\gk\left( \FF[x_i]\right) = \gk\left(\FF[y_j^{\pm 1}]\right) = 1 \leq 2$, the assertion follows by a repeated use of  Corollary~\ref{cor.tensor} and \cite[Proposition 3.12]{KraLen:gro}.
\end{proof}

\section{Differential smoothness of skew-polynomial rings}\label{sec.skew}
The aim of this section is to prove the following
\begin{theorem}\label{thm.skew}
Let $R$ be a  an algebra with an integrable differential calculus $(\Omega R,d)$ such that $\Omega R$ is a finitely generated right $R$-module. 
For any automorphism $\sigma$ of $R$ that extends to a degree-preserving automorphism of $\Omega R$, which commutes with $d$, there exists an integrable differential calculus $(\Omega A, d)$ on
the skew-polynomial ring $A=R[z;\sigma]$ and the Laurent skew-polynomial ring $A=R[z^{\pm 1};\sigma]$.
If $R$ is differentially smooth with respect to $(\Omega R,d)$ and $\gk(A)=\gk(R)+1$, then $A$ is also differentially smooth.
\end{theorem}

Recall that the {\rm trivial extension} of an algebra $A$ by an $A$-bimodule $M$ is the algebra $B$ isomorphic to $A\oplus M$ as a vector space and with the multiplication
$$
(a,m)(a',m') = (aa', am' + ma'), \qquad \mbox{for all $a,a'\in A$, $m,m'\in M$}.
$$
If $\nu$ is an automorphism of an algebra $A$, then we will denote by $A^\nu$, the $A$-bimodule with the multiplication
$$
a\cdot b \cdot a' := ab\nu(a'), \qquad \mbox{for all $a,a',b\in A$}.
$$
Furthermore, we write $M[z]$ (respectively $M[z^{\pm 1}]$) for the direct sum of identical copies of a bimodule $M$ labelled by all natural numbers (resp.\ integers), with the elements of the summand corresponding to $n$ written as $mz^n$, $m\in M$. As was the case of skew-polynomial rings $M[z^\bullet]$ denotes either $M[z]$ or $M[z^{\pm 1}]$.

\begin{lemma} \label{lem.extension}
Let $(\Omega R,d)$ be an $N$-dimensional differential calculus  on an algebra $R$ and let  $\sigma$ be a  degree-preserving automorphism of $\Omega R$ that commutes with $d$. Denote also by $\sigma$ the restriction of $\sigma$ to $R$, and let $S=\Omega R[z^\bullet;\sigma]$ and $A=R[z^\bullet;\sigma]$ be the corresponding skew-polynomial rings. Define the automorphism $\overline{\sigma}$ of $S$ by
\begin{equation}\label{bar.sigma} 
\overline{\sigma}(\omega z^n) = (-1)^{|\omega|}\sigma(\omega) z^n,
\end{equation}
for all homogeneous $\omega \in \Omega R$ and integers $n$.
Then the trivial extension $\Omega A = S\oplus S^{\overline{\sigma}}$ is an $N+1$-dimensional differential calculus on $R[z^\bullet;\sigma]$ with differential 
\begin{equation}\label{d} 
d(\omega z^n , \nu z^m) =  \left(d(\omega)z^n , (-1)^{|\omega|} \omega \partial_z(z^n) + d(\nu)z^m \right)
\end{equation}
for all homogeneous $\omega, \nu \in \Omega R$, where $\partial_z$ denotes the (formal) derivative of polynomials. The grading on $\Omega A$ is given by
\begin{equation}\label{grading}
|(\omega z^n , 0)| = |\omega|, \qquad |(0 , \nu z^m)| = |\nu| +1,
\end{equation}
for all homogeneous  $\omega, \nu \in \Omega R$ with $\nu\neq 0$.
\end{lemma}

\begin{proof}
Using the fact that the differential map $d$ in $\Omega R$ raises degree of a form by one, one easily checks that the map defined in \eqref{d} is square-zero. 
Note that, by equation \eqref{d},
$$
d(z, 0) = (0, 1),
$$
hence the generator $(0,1)$ of the $S$-bimodule $S^{\overline\sigma}\subset \Omega A$  can be denoted by $dz$, and 
$(\omega z^n, \nu z^m)$ can be interpreted as the differential form 
$\omega z^n + \nu z^m dz$.  Using this interpretation the equation  \eqref{d} comes out as
\begin{equation}\label{dd}
d(\omega z^n + \nu z^m dz) =  d(\omega)z^n + \left( (-1)^{|\omega|} \omega \partial_z(z^n) + d(\nu)z^m \right)dz.
\end{equation}
Furthermore, 
the multiplication in $S\oplus S^{\overline{\sigma}}$ says concretely that, for all $\nu\in\Omega^kR$,
$$ (0,1)(\omega z^n, 0) = (0, \overline{\sigma}(\omega z^n)) =  (0, (-1)^{|\omega|}\sigma(\omega)z^n),
$$
meaning 
\begin{equation}\label{dz.omega}
dz   \omega z^n = (-1)^{|\omega|}\sigma(\omega )  z^n dz = \overline{\sigma}(\omega z^n) dz. 
\end{equation}
The structure of a trivial extension pays tribute to the fact that $\Omega A dz$ is a square-zero ideal of the algebra of differential forms $\Omega A$, hence 
\begin{equation}\label{dz}
dz   dz =0.
\end{equation}
The equations \eqref{dz.omega} and \eqref{dz} determine fully the algebra structure of $\Omega A$. 

We need to check that the map $d$ defined by \eqref{dd}  satisfies the graded Leibniz rule. Let us take any homogeneous $\omega, \nu \in \Omega R$ and compute
\begin{eqnarray*}
d(\omega z^n    \nu z^m) &=& d(\omega\sigma^n(\nu) z^{n+m})\\
&=& d(\omega\sigma^n(\nu))z^{n+m} + (-1)^{|\omega|+|\nu|} \omega \sigma^n(\nu)\partial_z(z^{n+m}) dz\\
&=& d(\omega)\sigma^{n}(\nu)z^{n+m} + (-1)^{|\omega|}\omega \sigma^n(d(\nu))z^{n+m} \\
&&  +  (-1)^{|\omega|+|\nu|} \omega \sigma^n(\nu) (\partial_z(z^n)z^m + z^n\partial_z(z^m)) dz\\
&=& d(\omega)z^{n}\nu z^{m} + (-1)^{|\omega|}\omega z^{n} d(\nu)z^m \\
&&  +  (-1)^{|\omega|+|\nu|} (\omega \partial_z(z^n) \sigma(\nu) z^m + \omega z^n \nu \partial_z(z^m)) dz\\
&=& \left( d(\omega)z^{n}+  (-1)^{|\omega|} \omega \partial_z(z^n)dz \right) \nu z^m  \\
&&  + (-1)^{|\omega|}\omega z^{n} \left( d(\nu)z^m  +  (-1)^{|\nu|}  \nu \partial_z(z^m) dz\right)\\
&=& d(\omega z^n)\nu z^m + (-1)^{|\omega|} \omega z^n d(\nu z^m),
\end{eqnarray*}
where we used the definition of multiplication of the skew-polynomial algebra, \eqref{dd} and the fact that both $d$ on $\Omega R$ and $\partial_z$ satisfy the (graded) Leibniz rule. Next
\begin{eqnarray*}
d(\omega z^n dz \:\: \nu z^ m) &=& (-1)^{|\nu|}d(\omega \sigma^{n+1}(\nu) z^{n+m} dz )\\
&=& (-1)^{|\nu|}  d(\omega \sigma^{n+1}(\nu))z^{n+m} dz\\
&=& (-1)^{|\nu|} [ d(\omega)\sigma^{n+1}(\nu) + (-1)^{|\omega|}\omega d(\sigma^{n+1}(\nu))] z^{n+m}dz \\
&=& d(\omega z^n) \overline{\sigma}(\nu z^m) dz  + (-1)^{|\omega|} \omega z^n \overline{\sigma}(d(\nu z^m)) dz \\
&=& d(\omega z^n dz) \:\: \nu z^m  + (-1)^{|\omega|} \omega z^n dz\:\: d(\nu z^ m) ,
\end{eqnarray*}
by \eqref{dz.omega} and \eqref{dd}. Finally we can compute: 
\begin{eqnarray*}
d(\omega z^n \:\: \nu z^ m dz) &=& d(\omega \sigma^{n}(\nu) z^{n+m} dz )\\
&=&   d(\omega \sigma^{n}(\nu))z^{n+m} dz\\
&=&  [d(\omega)\sigma^{n}(\nu) + (-1)^{|\omega|}\omega \sigma^{n}(d(\nu))] z^{n+m}dz \\
&=& d(\omega z^n) \nu z^m dz  + (-1)^{|\omega|} \omega z^n d(\nu z^m dz).
\end{eqnarray*}
This proves that $(\Omega A, d)$ is a differential graded algebra. It is clear that $\Omega^{N+1}A = \Omega ^NR[z^\bullet] dz \neq 0$ and there are no components $\Omega^n A$ if $n>N+1$, hence $(\Omega A, d)$ has dimension $N+1$. Since every element of $(\Omega A, d)$ can be written as a linear combination of $\omega z^n + \nu z^m dz$, with $\omega,\nu \in \Omega R$ and $\Omega R$ satisfies the density condition (over $R$), also $\Omega A$ satisfies this condition (over $A$). Therefore, $(\Omega A, d)$ is an $N+1$-dimensional calculus as claimed.
\end{proof}

\begin{lemma}\label{lem.volume}
In the set-up of Lemma~\ref{lem.extension} assume that $(\Omega R, d)$ has a volume $N$-form $\topform$ with the twisting automorphism $\theta_\v$ and the co-ordinate isomorphism $\pi_\v$. 
Let $u = \pi_\v (\sigma(\v))$ and define the map
\begin{equation}\label{bar.theta}
\bar{\theta} : R[z^\bullet;\sigma] \to R[z^\bullet;\sigma], \qquad \sum_{i}a_iz^i\mapsto \sum_{i} \theta_\v(a_i)(uz)^i.
\end{equation}
Then:
\begin{zlist}
\item The map $\bar{\theta}$ is an algebra automorphism of $R[z^\bullet;\sigma]$.
\item $\Omega \left(R[z^\bullet;\sigma]\right)$ has a volume form $\topform dz$ with the twisting automorphism $\overline{\sigma}^{-1}\circ\bar{\theta}$.
 \end{zlist}
\end{lemma} 

\begin{proof}
First note that  there exists  $v\in R$ such that $\sigma^{-1}(\topform)=\topform v$. Hence $u\sigma(v) = 1 = v\sigma^{-1}(u)$, i.e. $u$ is invertible and the map $\bar{\theta}$ is well defined also in the Laurent case. Furthermore  $\bar{\theta}$ is invertible,  with the inverse
$$
\bar{\theta}^{-1} : R[z;\sigma] \to R[z;\sigma], \qquad \sum_{i} a_iz^i\mapsto \sum_{i} \theta_\v^{-1}(a_i)(u^{-1}z)^i.
$$

Since, for all $a\in R$,  $a\topform  = \topform \theta_\v(a)$ and
$\sigma(\topform)= \topform \pi_\v (\sigma(\v)) = \topform u$,
$$
\topform\theta_\v(\sigma(a))u = \sigma(a)\topform u = \sigma(a)\sigma(\topform)=\sigma(a\topform)=\sigma(\topform\theta_\v(a))=\sigma(\topform)\sigma(\theta_\v(a))=\topform u \sigma(\theta_\v(a)).
$$
Hence
$$ 
\theta_\v(\sigma(a))u  =  u \sigma(\theta_\v(a))
$$
holds for all $a$ in $R$, and therefore
$$
\bar{\theta}(z) \bar{\theta}(a) = uz\theta_\v(a) = u\sigma(\theta_\v(a)) z =  \theta_\v(\sigma(a))uz = \bar{\theta}(\sigma(a))\bar{\theta}(z),
$$
and
\begin{eqnarray*}
\bar{\theta}(z^{-1}) \bar{\theta}(a) &=& z^{-1}u^{-1}\theta_\v(a) = z^{-1}\sigma(\theta_\v(\sigma^{-1}(a))) u^{-1}\\
& = & \theta_\v(\sigma^{-1}(a))z^{-1}u^{-1} = \bar{\theta}(\sigma^{-1}(a))\bar{\theta}(z^{-1}),
\end{eqnarray*}
This implies that $\bar{\theta}$ is an algebra map and completes the proof of the first assertion.

To prove the second assertion, first let us write $A$ for $R[z^\bullet;\sigma]$. As $\Omega^N R = \topform R$, we have $\left(\Omega^N R\right)[z^\bullet] = \topform A$, and thus
$\Omega^{N+1}A = \left(\Omega^N R\right)[z^\bullet] dz = \topform A dz$, i.e.\ any element of $\Omega^{N+1}A$ is of the form $\topform f dz$, for some $f\in A$.
By \eqref{dz.omega}, 
$\topform f dz =  \topform dz  \overline{\sigma}^{-1}(f)$, hence $\Omega^{N+1} A = \topform dz A$. Let $f=\sum_{} a_i z^i \in A$ be such that $\topform dz f = 0$, then $\topform \overline{\sigma}(f) = 0$, which implies  $f=0$. Thus $\Omega^{N+1}A = \topform dz A$ is a free rank one right $A$-module.
Moreover, for any element $f \in A$,
$$ 
f \topform dz = \topform \bar{\theta}(f)  dz = \topform dz \overline{\sigma}^{-1}(\bar{\theta}(f)),
$$
which also shows that $A\topform dz = \topform dz A$, and hence $\topform dz$  is a free generator on both sides, and that the twisting automorphism has the stated form. 
\end{proof}

\begin{lemma}\label{lem.morphisms}
Let $R\subseteq S$ be a ring extension and $\sigma \in \mathrm{Aut}(S)$ such that the restriction of $\sigma$ to $R$ is an automorphism of $R$. Consider the skew-polynomial rings $S[z^\bullet;\sigma]$ and its subring $R[z^\bullet;\sigma]$.
Let $M$ be a $\sigma$-stable, right $R$-submodule of $S$ and consider the right $R[z^\bullet;\sigma]$-submodule $M[z^\bullet]$ of $S[z^\bullet;\sigma]$.

\begin{enumerate}
	\item The additive map $\psi: \Hom_R(M,R)[z^\bullet]\rightarrow \Hom_{R[z;\sigma]}\left(M[z^\bullet] ,R[z^\bullet;\sigma]\right)$ given by 
	$$ fz^k \mapsto \psi_{fz^k}:[mz^i\mapsto \sigma^k(f(m))z^{k+i}],\qquad \forall f\in \Hom_R(M,R),
	$$ is  well-defined and injective.
	\item If $M$ is a finitely generated right $R$-module and $\sigma(M)=M$, then $\psi$ is bijective.
\end{enumerate}
\end{lemma}

\begin{proof}
	(1) We will first show that $\psi_{fz^k}$ is a right $R[z^\bullet;\sigma]$-linear map. For all $m\in M$, $r\in R$ and $i,j \in \ZZ$,
	\begin{eqnarray*}
	\psi_{fz^k}(mz^i)rz^j &=&	\sigma^k(f(m)) z^{k+i} rz^j \\
	 &=&	\sigma^k(f(m\sigma^{i}(r))) z^{k+i+j}
	 =	\psi_{fz^k}\left(m\sigma^{i}(r)z^{i+j}\right) = \psi_{fz^k}\left(mz^i rz^j\right)
	\end{eqnarray*}
If $\psi_{\sum_{k} f_kz^k}=0$ for some $\sum_{k} f_kz^k \in \Hom_R(M,R)[z^\bullet]$, then, for all $m\in M$,
$$
\sum_{k} \sigma^k\left(f_k\left(m\right)\right) z^k = \psi_{\sum_{k} f_kz^k}(m)=0,
$$ 
i.e.\ $\sigma^k(f_k(m))=0$, for all $ k$. Hence $f_k=0$ for all $k$, showing that $\psi$ is injective.
	
	(2) For any $k$, define an additive map  
$$
p_k:R[z;\sigma]\rightarrow R, \qquad 
p_k\left(\sum r_iz^i\right)=\sigma^{-k}(r_k).
$$
The maps $p_k$ are  right $R$-linear, because
	$$
	p_k\left(\sum r_iz^i r'\right)=p_k\left(\sum r_i\sigma^i(r')z^i\right)=\sigma^{-k}(r_k)r'=p_k\left(\sum r_iz^i\right)r'.
	$$	
	
	Suppose $\{b_1,\ldots, b_n\}$ is a generating set for $M_R$. Let 
	$f\in \Hom_{R[z^\bullet;\sigma]}\left(M[z^\bullet],R[z^\bullet;\sigma]\right)$. There exist a finite indexing set $I\subset \ZZ$ and elements $\beta_{ik} \in R$, for $1\leq i \leq n$ and $k \in I$, such that
	$f(b_i) = \sum_{k\in I} \beta_{ik}z^k.$ For all $k \in I$, define right $R$-linear maps $f_k \in \Hom_R(M,R)$ by the composition
	$f_k = p_k \circ f$, i.e. $f_k(b_i)=\sigma^{-k}(\beta_{ik}),$ for all $i$.
	Let $m\in M$. There are $\lambda_i\in R$ such that $m=\sum_{i=1}^n b_i\lambda_i$. Then, for any $j$,
\begin{eqnarray*}
f(mz^j)&=&\sum_{i=1}^n f(b_i)\lambda_i z^j 
= \sum_{k\in I} \sum_{i=1}^n \beta_{ik}z^k \lambda_i z^j\\
&=& \sum_{k\in I} \sum_{i=1}^n \sigma^k(f_k(b_i)) \sigma^k(\lambda_i )z^{k+j}
= \sum_{k\in I} \sigma^k(f_k(m))z^{k+j} = \psi_{\sum_{k\in I} f_kz^k}\left(mz^j\right).
\end{eqnarray*}
Hence $f= \psi\left( \sum_{k\in I} f_kz^k\right).$
\end{proof}

\begin{corollary}\label{isomorphism}
	Let $\sigma$ be an automorphism of degree $0$ of a graded algebra $\Omega=\bigoplus_{k=0}^\infty \Omega^k$. Set $R=\Omega^0$. If $\Omega^k$ is finitely generated as right $R$-module, then
$$
	\Hom_R(\Omega^k,R)[z^\bullet] \simeq \Hom_{R[z^\bullet;\sigma]}(\Omega^k[z^\bullet],R[z^\bullet;\sigma]).
$$
\end{corollary}

With these assertions at hand we can now prove Theorem~\ref{thm.skew}.

\begin{proofof}{Theorem~\ref{thm.skew}}
Let us denote by $\topform$ a volume form for $\Omega R$, with the corresponding co-ordinate isomorphism $\pi_\v: \Omega^N R \rightarrow R$. 
By Lemma~\ref{lem.volume},  $\topform dz$ is a volume form for the differential calculus  $\Omega A$ on $A=R[z^\bullet, \sigma]$ constructed in Lemma~\ref{lem.extension}, and let  $\pi_{\v dz}: \Omega^{N+1} A \rightarrow A$ be the corresponding co-ordinate isomorphism. For all $\omega \in \Omega^NR$,
\begin{equation}\label{eq_pi_pi} 
\pi_{\v dz}(\omega z^i dz) = \pi_{\v dz}\left(\topform \pi_\v(\omega) dz\right)z^i  = \sigma^{-1}(\pi_\v(\omega)) z^i.\end{equation}
Consider the maps $\ell_\v^k : \Omega^k R \rightarrow \Hom_R(\Omega^{N-k}R, R)$, $\ell_{\v dz}: \Omega^k A \rightarrow \Hom_A(\Omega^{N+1-k}A, A)$  associated to respective volume forms by \eqref{ell}. 
 We   canonically extend $\ell_\v^k$ 
 to maps from $\Omega^kR[z^\bullet]$ to $\Hom_R(\Omega^{N-k}R, R)[z^\bullet]$, by acting on the coefficients.  To show that $A$ is differentially smooth, we need to show that the maps $\ell_{\v dz}^k:\Omega^kA \rightarrow \Hom_{A}\left(\Omega^{N-k}A ,A\right)$ are bijective. For $k=0$ or $k=N+1$ these bijections are clear since $\v dz$ is a volume form for $\Omega A$. Let $1\leq k \leq N$ and recall that $\Omega^k A = (\Omega^k R)[z^\bullet] \oplus (\Omega^{k-1}R)[z^\bullet] dz$.  It is not difficult to see that the image of $(\Omega^k R)[z^\bullet]$ under $\ell_{\v dz}^k$ lies in $ \Hom_A(\left(\Omega^{N-k}R\right)[z^\bullet]dz, A)$,
while the image of $(\Omega^{k-1}R)[z^\bullet] dz$ under $\ell_{\v dz}^k$ lies in  $\Hom_A(\left(\Omega^{N+1-k}R\right)[z^\bullet], A)$. 

For later use we define for any $i\in \ZZ$ invertible elements   $v_i\in R$, such that $\sigma^{i}(\v)=v_i \v$.
Then $v_{i+j}\v = \sigma^{i}(\sigma^{j}(\v))
=\sigma^{i}(v_j)v_i\v$ for any $i,j \in \ZZ$, i.e. $v_{i+j}=\sigma^{i}(v_j)v_i$ from which it follows that $\sigma^{i}(v_{-i})$ is the inverse of $v_i$.
The following equation shows the commutation of $\sigma$ and $\pi_\v$. Let $\omega \in \Omega^N R$. Then
\begin{equation}\label{eq:comm:sigma:pi}
\sigma^{-i}(\pi_\v (\omega)) = \pi_\v (\v \sigma^{-i}(\pi_\v (\omega)))
 = \pi_\v (\sigma^{-i}( v_{i}\v \pi_\v (\omega)))
  = \pi_\v (\sigma^{-i}( v_{i}\omega)).
\end{equation}

Furthermore, the invertible elements $v_i$ give rise to a linear isomorphism 
$$\Phi: (\Omega^{k-1}R)[z^\bullet]dz \rightarrow (\Omega^{k-1}R)[z^\bullet],
\qquad  \omega z^i dz \mapsto (-1)^{N-k+1}\sigma^{-(i+1)} (v_{i+1} \omega )z^i .$$

Using equations  (\ref{eq_pi_pi}) and (\ref{eq:comm:sigma:pi})  we compute, for all $\omega z^i\in \Omega^{k-1}R[z^\bullet]$ and $ \omega' z^j\in \Omega^{N-k+1}R[z^\bullet]$,
\begin{eqnarray*}
\ell_{\v dz}^{k}(\omega z^i dz)(\omega' z^j)&=&
\pi_{\v dz}(\omega z^i dz\omega' z^j) 
=(-1)^{|\omega'|} \pi_{\v dz}(\omega \sigma^{i+1}(\omega') z^{i+j} dz) \\
&=& (-1)^{|\omega'|} \sigma^{-1}(\pi_\v(\omega \sigma^{i+1}(\omega'))) z^{i+j} \\
&=& (-1)^{|\omega'|} \sigma^i(\sigma^{-(i+1)}(\pi_\v(\omega \sigma^{i+1}(\omega')))) z^{i+j} \\
&=& (-1)^{|\omega'|} \sigma^i(\pi_\v(\sigma^{-(i+1)}(v_{i+1}\omega \sigma^{i+1}(\omega')))) z^{i+j} \\
&=& (-1)^{|\omega'|} \sigma^i(\pi_\v(\sigma^{-(i+1)}(v_{i+1}\omega) \omega') z^{i+j} \\
&=& \psi ( \ell_\v^{k-1}(  (-1)^{|\omega'|} \sigma^{-(i+1)}(v_{i+1}\omega z^i) )) (  \omega' z^{j})\\
&=& \psi ( \ell_\v^{k-1}( \Phi( \omega z^i dz))) (  \omega' z^{j}),
\end{eqnarray*}
where $\psi: \Hom_R(\Omega^{N-k+1}R,R)[z^\bullet]\rightarrow \Hom_{A}\left(\Omega^{N-k+1}R[z^\bullet] ,A\right)$ is the homomorphism from Lemma \ref{lem.morphisms}. 
Hence we have shown that the following diagram commutes:
\begin{equation}\label{diagram1}\begin{CD}
 \left(\Omega^{k-1}R\right)[z^\bullet]dz @>{\ell_{\v dz}^k}>>  \Hom_A(\Omega^{N+1-k}R[z^\bullet], A) \\
 @V{\Phi}VV @AA{\psi}A \\
\left(  \Omega^{k-1}R\right)[z^\bullet] @>>{\ell_{\v}^{k-1}}>  \Hom_R(\Omega^{N+1-k}R,R)[z^\bullet].
\end{CD}
\end{equation}

In a way similar to  the definition of $\Phi$ we define the linear isomorphism
$\Phi':  \left(\Omega^{k}R\right)[z^\bullet] \rightarrow  \left(\Omega^{k}R\right)[z^\bullet]$ by
$$
\Phi'(\omega z^i) = \sigma^{-(i+1)}(v_{i+1}\omega) z^i  \qquad  \mbox{ for all } \omega \in \Omega^k R.
$$  
Moreover, the map $\varphi: \Omega^{N-k}R[z^\bullet]dz \rightarrow \Omega^{N-k}R[z^\bullet]$, given by $\varphi(\omega z^i dz) = \sigma^{-1}(\omega)z^i$, is an isomorphism of right $A$-modules. 
The adjoint map of $\varphi$ is  the isomorphism
$$
\varphi^*:  \Hom_A\left(  \Omega^{N-k}R\right)[z^\bullet], A) \rightarrow \Hom_A(\Omega^{N-k}R[z^\bullet]dz, A),
$$
$$
f \mapsto f\circ \varphi : [\omega z^idz \mapsto f(\sigma^{-1}(\omega)z^i)].
$$

Using again equation (\ref{eq:comm:sigma:pi}), $\Phi'$ and $\varphi^*$ we compute, for all $\omega z^i \in \Omega^{k} R[z^\bullet], \omega' z^j\in \Omega^{N-k} R[z^\bullet]$,
\begin{eqnarray*}
\ell_{\v dz}^{k}(\omega z^i)(\omega' z^j dz)&=&
\pi_{\v dz}(\omega z^i \omega' z^j dz) 
=\pi_{\v dz}(\omega\sigma^i(\omega') z^{i+j} dz)  \\
&=& \sigma^{-1}(\pi_\v(\omega\sigma^i(\omega'))) z^{i+j} 
= \sigma^i(\sigma^{-(i+1)}(\pi_\v(\omega\sigma^i(\omega')))) z^{i+j} \\
&=& \sigma^i(\pi_\v(\sigma^{-(i+1)}(v_{i+1}\omega\sigma^i(\omega'))) z^{i+j} \\
&=& \sigma^i(\pi_\v(\sigma^{-(i+1)}(v_{i+1}\omega)\sigma^{-1}(\omega'))) z^{i+j} \\
&=& \psi(\ell_\v^k( \sigma^{-(i+1)}(v_{i+1}\omega) z^i)) (\sigma^{-1}(\omega')z^j)\\
&=& \psi(\ell_\v^k( \Phi'(\omega z^i))) (\varphi( \omega'z^j dz))
= \varphi^*(\psi(\ell_\v^k( \Phi'(\omega z^i)))) ( \omega'z^j dz).
\end{eqnarray*}
Therefore, the following diagram commutes:
\begin{equation}\label{diagram2}\begin{CD}
 \left(\Omega^{k}R\right)[z^\bullet] @>{\ell_{\v dz}^k}>>  \Hom_A(\Omega^{N-k}R[z^\bullet]dz, A) \\
 @V{\Phi'}VV @AA{\varphi^*\circ\psi}A \\
 \left(\Omega^{k}R\right)[z^\bullet] @>>{\ell_{\v}^k}>  \Hom_R(\Omega^{N-k}R,R)[z^\bullet] . 
\end{CD}\end{equation}

Assume for all $k$, that the maps $\ell_{\v}^k$ are bijective and that the right $R$-modules $\Omega^k R$ are finitely generated. Then $\psi$ is bijective by Corollary \ref{isomorphism} and hence $\psi\circ\ell_{\v}^{k-1}\circ\Phi$ is bijective as well as $\varphi^*\circ\psi\circ \ell_{\v}^k\circ \Phi'$ are bijective maps. Since the diagrams (\ref{diagram1}) and (\ref{diagram2}) commute, also $\ell_{\v dz}^k:\Omega^k A \rightarrow  \Hom_A(\Omega^{N+1-k}A, A) $ is bijective.

Finally, if $\Omega R$ is a connected calculus then vanishing of the first component in \eqref{d}, implies that if $d(f) =0$, for $f =\sum_i a_iz^i \in R[z^\bullet;\sigma]$, then  $f\in \FF[z^\bullet]$ (i.e.\ $f$ has scalar coefficients only). The  second component in \eqref{d} is simply $\partial_z (f)$, hence it vanishes if and only if $f$ is a scalar multiple of the identity in $A$. Therefore, the calculus $\Omega A$ is also connected. This completes the proof of the theorem.
\end{proofof}

\begin{example}\label{ex.flip}
For any non-zero $q\in \FF$, let us define $A_q$ as an algebra generated by $x,y,z$ and relations
\begin{equation}\label{flip}
xy=yx, \qquad xz=q zy, \qquad yz = zx.
\end{equation}
The algebra $A_q$ is differentially smooth. Similarly, the algebra $B_q$, generated by  $x$, $y$ and invertible $z$  subject to relations \eqref{flip}, is differentially smooth.
\end{example}
\begin{proof}
The algebras $A_q$ and $B_q$ are both
skew-polynomial rings, $A_q = \FF[x,y][z;\sigma]$, $B_q = \FF[x,y][z^{\pm 1};\sigma]$, where the automorphism $\sigma$ of $\FF[x,y]$ is given by
$$
\sigma(x) = y, \qquad \sigma(y) = q x.
$$
The polynomial algebra $\FF[x,y]$ is differentially smooth with the usual commutative differential calculus $\Omega(\FF[x,y])$,
$$
xdx = dxx, \qquad xdy = dyx, \qquad ydx = dxy, \qquad ydy=dyy, 
$$
$$
 dxdy = - dydx, \qquad (dx)^2 = (dy)^2 =0.
$$
The automorphism $\sigma$ extends to the automorphism of $\Omega(\FF[x,y])$ by requesting  it commute with $d$, i.e.\
$$
\sigma(dx) =  dy, \qquad \sigma(dy) =  qdx.
$$
Since $\Omega(\FF[x,y])$ is finitely generated as a right $\FF[x,y]$-module and
$$
\gk(A_q) = \gk(B_q) = 3 = \gk(\FF[x,y]) +1,
$$
Theorem~\ref{thm.skew} yields the differential smoothness of $A_q$ and $B_q$. 
\end{proof}

\begin{remark}\label{rem.flip}
We notice in passing that $B_1$ in Example~\ref{ex.flip} contains the down-up algebra $A(0,1,0)$ \cite{BenRob:dow} as a proper subalgebra and hence the assertion of Example~\ref{ex.flip} can be a starting point in determining whether $A(0,1,0)$  is differentially smooth.
\end{remark}

The statement of Theorem~\ref{thm.skew} can be iterated in the following way.

\begin{proposition}\label{prop.iteration}
Let $R$ be an algebra with an integrable differential calculus $(\Omega R,d)$ such that $\Omega R$ is a finitely generated right $R$-module. Let $\sigma$ be an automorphism of $R$ that extends to a degree-preserving automorphism of $\Omega R$, which commutes with $d$. Let $(\Omega A, d)$  be the integrable differential calculus on  $A=R[z^\bullet;\sigma]$ constructed via Theorem~\ref{thm.skew}. 
\begin{zlist}
\item   For any $q\in \FF^*$, the map $\sigma$ extends to an automorphism of  the differential graded algebra $(\Omega A, d)$, by
$$
\sigma_q: \Omega A \to \Omega A, \qquad \omega f(z) \mapsto \sigma(\omega) f(qz), \quad  \omega f(z)dz  \mapsto q\sigma(\omega) f(qz)dz.
$$
\item 
If $R$ is differentially smooth with respect to $(\Omega R,d)$ and $\gk(A)=\gk(R)+1$, then $A[t^\bullet;\sigma_q] = R[z^\bullet;\sigma][t^\bullet;\sigma_q]$ is also differentially smooth.
\end{zlist}
\end{proposition}
\begin{proof}
That $\sigma_q$ is an algebra automorphism is established by a routine calculation. To check that $\sigma_q$ commutes with $d$, first observe that 
\begin{equation}\label{q-com}
\partial_z \circ \sigma_q = q\,\sigma_q\circ \partial_z.
\end{equation}
Hence, for all $f,g\in \FF[z^\bullet]$ and homogeneous $\omega, \nu \in \Omega R$,
\begin{eqnarray*}
\sigma_q \left(d \left(\omega f +\nu g dz\right)\right) &=& \sigma(d\omega)\sigma_q(f) + q\left((-1)^{|\omega|}\sigma(\omega)\sigma_q(\partial_z(f)) +\sigma(d\nu)\sigma_q(g)\right)dz\\
&=& d\left(\sigma(\omega)\right)\sigma_q(f) + (-1)^{|\omega|}\sigma(\omega)\partial_z(\sigma_q(f))dz +d\left(\sigma(\nu)\right)\sigma_q(g dz)\\
&=& d\left(\sigma_q \left(\omega f +\nu g dz\right) \right),
\end{eqnarray*}
where the second equality follows by \eqref{q-com} and the fact that $\sigma$ commutes with $d$. This completes the proof of the first statement. 

Since $\Omega R$ is finitely generated as a right $R$ module and 
$$
\Omega^k A = \Omega^kR[z^\bullet;\sigma]\oplus \Omega^{k-1}R[z^\bullet;\sigma]^{\bar\sigma},
$$
also $\Omega A$ is finitely generated as a right $A$-module.  $\Omega A$ is integrable of dimension $\gk(R) +1$, hence, by Theorem~\ref{thm.skew} $A[t^\bullet;\sigma_q]$ admits an integrable calculus of dimension $\gk(R) + 2$. Furthermore, the automorphism $\sigma_q$ is locally algebraic, hence $\gk(A[t^\bullet;\sigma_q]) = \gk(A) +1$ by \cite[Proposition~1]{LerMat:Gel}, and since $\gk(A)=\gk(R)+1$ the second assertion follows.
\end{proof}

Proposition~\ref{prop.iteration} leads to a quick proof of the differential smoothness of special cases of algebras whose differential smoothness was established in \cite{KarLom:int}.

\begin{corollary}\label{cor.affine}
The coordinate algebra of the non-commutative $n$-dimensional affine space, i.e.\ the algebra $\FF_{\bf q}[x_1,\ldots, x_n]$ generated by $x_1,\ldots, x_n$ subject to the relations
$$
x_ix_j = q_ix_jx_i, \qquad \mbox{for all $i<j$}, 
$$
where $\mathbf{q} = (q_1, \ldots , q_{n-1})\in (\FF^*)^{n-1}$, is differentially smooth.
\end{corollary}
\begin{proof}
$\FF_{\bf q}[x_1,\ldots, x_n]$ is an iterated skew polynomial ring. Starting with the polynomial ring $\FF[x_1]$, which is differentially smooth by the usual commutative differential structure, and applying first Theorem~\ref{thm.skew} and then its iteration Proposition~\ref{prop.iteration} sufficiently many times (with a different $q$ at each step), we conclude that $\FF_{\bf q}[x_1,\ldots, x_n]$ is differentially smooth, as claimed.
\end{proof}

 \section*{Acknowledgments}
This work was done during a visit of the second author to Swansea University in the framework of his sabbatical leave. He would like to thank the first author and all members of the Department of Mathematics for their warm hospitality. Moreover, the second author thankfully acknowledges the financial support of the College of Science Research Fund of Swansea University and the award of the LMS visitor grant (scheme 2), Ref. No. 21501.

\end{document}